\documentclass{amsart}
\usepackage{amsmath, amssymb,verbatim}
\newtheorem{thm}{Theorem}

\theoremstyle{definition}

\newtheorem{remark}{Remark}
\author{Khitam Aqel \and Issam Louhichi}
\address{Department of Mathematics \& Statistics, College of Arts \& Sciences, American University of Sharjah, P.O.Box 2666, Sharjah, UAE.}
\begin{document}
\title[On the commutativity of sums of Toeplitz operators ]
{On the commutativity of sums of Toeplitz operators on the Bergman
space}
\email{g00047623@aus.edu}
\email{ilouhichi@aus.edu}
\date{\today} 
\begin{abstract}
In this paper, we discuss the commutativity of sums of two quasihomogeneous Toeplitz operators on the Bergman space of the unit disc. Our main result goes in the direction of the conjecture in \cite[p.~263]{lr}.	
\end{abstract}	
\maketitle
Let $\mathbb{D}$ be the unit disc of the complex plane $\mathbb{C}$. We denote by $L^2_a$, and we call it the usual unweighted Bergman space, the Hilbert space of analytic functions on $\mathbb{D}$ that are square integrable with respect to the normalized Lebesgue measure $dA(z)=rdr\frac{d\theta}{\pi}$, where $(r,\theta)$ are the polar coordinates. $L^2_a$ is a closed subspace of the Hilbert space $L^{2}(\mathbb{D},dA)$ and has the set $\{\sqrt{n+1}z^n\ |\ n=0, 1, 2,\ldots\}$ as an orthonormal basis. We denote by $P$ the orthogonal projection from $L^{2}(\mathbb{D},dA)$ onto $L^2_a$, called the Bergman projection. We define on $L^2_a$ the Toeplitz operator $T_h$ with symbol a bounded function $h$ by $T_hu=P(hu)$ for any $u\in L^2_a$.

One of the most intriguing questions for Toeplitz operators on the Bergman space is: When is the product (in a sense of composition)  of two Toeplitz operators commutative? Although a lot of work has been done on this question \cite{ac, acr, c, cr, l, lr, lra, lz, v}, we still have not reach a complete answer. 

The main motivation of this paper is the results of \cite{lr}. We shall show that under some assumptions if sums of two quasihomogeneous Toeplitz operators commute, then they are equal up to a multiplicative constant. A symbol $h$ is said to be quasihomogeneous of an integer degree $p$ if $h$ can be written as $h(re^{i\theta})=e^{ip\theta}\omega(r)$, where $\omega$ is a radial function in $\mathbb{D}$ i.e., $\omega(z)=\omega(|z|)$. In this case, the associated Toeplitz operator $T_h$ is also called quasihomogeneous Toeplitz operator of degree $p$. This class of operators got the interest of many people \cite{cr, l, lr, lr2, lra, lsz, lz} and have been largely studied since. 

Consider two bounded quasihomogeneous Toeplitz operators $T_{e^{ip\theta}\phi}$ and $T_{e^{is\theta}\psi}$ with $p, s\in\mathbb{N}$.  Assume that these two operators have roots $T_{e^{i\theta}\widetilde{\phi}}$ and $T_{e^{i\theta}\widetilde{\psi}}$ respectively (see \cite{l,lr2}) i.e.,
$$T_{e^{ip\theta}\phi}=\left(T_{e^{i\theta}\widetilde{\phi}}\right)^p
\textrm{ and } T_{e^{is\theta}\psi}=\left(T_{e^{i\theta}\widetilde{\psi}}\right)^s.$$
The purpose of this work is to characterize quasihomogeneous Toeplitz operators $T_{e^{im\theta}f}$ and $T_{e^{il\theta}g}$ where $m,l\in\mathbb{N}$,  such that
\begin{displaymath}
\left\{\begin{array}{ll}
(H1)& T_{e^{im\theta}f}+ T_{e^{il\theta}g}
\textrm{ commutes with } T_{e^{ip\theta}\phi}+T_{e^{is\theta}\psi},\\
(H2)& 1\leq p<s, \ 1\leq m<l, \textrm{ and } l+p=m+s.
\end{array}\right.
\end{displaymath}
Hypothesis $(H1)$ implies that for all $k\geq 0$, we have
$$\left(T_{e^{im\theta}f}+ T_{e^{il\theta}g}\right)\left(T_{e^{ip\theta}\phi}+T_{e^{is\theta}\psi}\right)(z^k)
=\left(T_{e^{ip\theta}\phi}+T_{e^{is\theta}\psi}\right)\left(T_{e^{im\theta}f}+ T_{e^{il\theta}g}\right)(z^k).$$
Now hypothesis  $(H2)$, combined with \cite[Remark~2]{lr}, implies that for all $k\geq 0$ we must have
\begin{eqnarray}
T_{e^{im\theta}f}T_{e^{ip\theta}\phi}(z^k)&=&T_{e^{ip\theta}\phi}T_{e^{im\theta}f}(z^k),\\
T_{e^{il\theta}g}T_{e^{is\theta}\psi}(z^k)&=&T_{e^{is\theta}\psi}T_{e^{il\theta}g}(z^k),\\
\left(T_{e^{im\theta}f}T_{e^{is\theta}\psi}+T_{e^{il\theta}g}T_{e^{ip\theta}\phi}\right)(z^k)
&=&\left(T_{e^{is\theta}\psi}T_{e^{im\theta}f}+T_{e^{ip\theta}\phi}T_{e^{il\theta}g}\right)(z^k).
\end{eqnarray}

(1) and (2) imply that $T_{e^{im\theta}f}$ (resp. $T_{e^{il\theta}g}$) commutes with $T_{e^{ip\theta}\phi}$ (resp. $T_{e^{is\theta}\psi}$). Therefore, using \cite[Proposition~2 and Lemma ~2]{lr}, we obtain that
\begin{equation}\label{root_f}
 T_{e^{im\theta}f}=c_1\left(T_{e^{i\theta}\widetilde{\phi}}\right)^m,
\end{equation}
 and
\begin{equation}\label{root_g}
 T_{e^{il\theta}g}=c_2\left(T_{e^{i\theta}\widetilde{\psi}}\right)^l,
\end{equation}
for some constants $c_1$ and $c_2$. To avoid the trivial case, which is $ T_{e^{im\theta}f}$ and $T_{e^{il\theta}g}$ being the zero operator, we assume that   $c_1$  and $c_2$ are nonzero constants. Thus Equation (3) can be written as 
\begin{equation}\label{com}
c_1\left[\left(T_{e^{i\theta}\widetilde{\phi}}\right)^m,T_{e^{is\theta}\psi}\right](z^k)=c_2\left[T_{e^{ip\theta}\phi},\left(T_{e^{i\theta}\widetilde{\psi}}\right)^l\right](z^k)\textrm{ for all }k\geq 0,
\end{equation}
where $[A,B]=AB-BA$ denotes the commutator of the operators $A$ and $B$. 

\begin{remark}
\begin{itemize}
\item[i)] If $T_{e^{ip\theta}\phi}$ and $T_{e^{is\theta}\psi}$ commute with each other, then \cite[Proposition~2 and Lemma ~2]{lr} imply that $T_{e^{i\theta}\widetilde{\psi}}= c T_{e^{i\theta}\widetilde{\phi}}$ for some constant $c$. Moreover, \cite[Corollary ~1]{lr} implies that all four Toeplitz operators $T_{e^{ip\theta}\phi}$, $T_{e^{is\theta}\psi}$, $T_{e^{im\theta}f}$ and $T_{e^{il\theta}g}$ commute with each other, and therefore they are all of the form constant times power of a single Toeplitz operator $T_{e^{i\theta}\widetilde{\phi}}$. So without loss of generality, we assume $[T_{e^{ip\theta}\phi}, T_{e^{is\theta}\psi}]\neq 0$ from now on.
\item[ii)] The case $p=s$ (resp. $l=m$) has been extensively studied and totally solved. See \cite{cr, l, lz}.
\item[iii)] We shall show that for a certain class of Toeplitz operators $T_{e^{ip\theta}\phi}$, $T_{e^{is\theta}\psi}$ if $(H1)$ and $(H2)$ hold, then $m=p$, $l=s$, and hence $c_1	=c_2$. In other words $T_{e^{im\theta}f}+ T_{e^{il\theta}g}$ is simply  constant times  $T_{e^{ip\theta}\phi}+T_{e^{is\theta}\psi}$. In fact, if $m=p$ (resp. $l=s$), then $(H2)$ implies $l=s$ (resp. $m=p$). Moreover, equations (\ref{root_f}) and (\ref{root_g}) imply $T_{e^{im\theta}f}=c_1T_{e^{ip\theta}\phi}$ and $T_{e^{il\theta}g}=c_2T_{e^{is\theta}\psi}$, for some constants $c_1, c_2$. Thus Equation (\ref{com}) becomes
$$c_1\left[T_{e^{ip\theta}\phi},T_{e^{is\theta}\psi}\right](z^k)=c_2\left[T_{e^{ip\theta}\phi},T_{e^{is\theta}\psi}\right](z^k)\textrm{ for all }k\geq 0,$$
and therefore $c_1=c_2$ since we assume that  $[T_{e^{ip\theta}\phi}, T_{e^{is\theta}\psi}]\neq 0$.
\end{itemize}
\end{remark}
Quasihomogeneous Toeplitz operators have an interesting property which is acting on the elements of the orthogonal basis of $L^2_a$ as shift operators with holomorphic weight \cite{lr}. In fact if $\phi$ is a bounded radial function and $p$ a positive integer, then for all $k\geq 0$, we have
\begin{equation*} 
\begin{split}
T_{e^{ip\theta}\phi}(\zeta^k) (z)& =\displaystyle{ \int_{0}^{1}\int_{0}^{2\pi}\phi(r)r^k\sum\limits_{j=0}^{\infty}(j+1)e^{i(k+p-j)\theta}r^jz^j\frac{1}{\pi}rdr d\theta}\\
& = 2(k+p+1)\int_{0}^{1}\phi(r)r^{2k+p+1}dr\ z^{k+p}.
\end{split}
\end{equation*}
Now, we define the Mellin transform of a function $\phi$ in $L^{1}([0,1],rdr)$ by
\begin{equation*}
\mathcal{M}\left(\phi\right)(z)=\int_{0}^{1}\phi(r)r^{z-1}dr.
\end{equation*}
It is easy to see that $\mathcal{M}\left(\phi\right)$ is bounded and holomorphic in the right-half plane $\{z\in \mathbb{C}|\Re z>2\}$. Thus
\begin{equation*}
T_{e^{ip\theta}\phi}(\zeta^k) (z)=2(k+p+1)\mathcal{M}\left(\phi\right)(2k+p+2)z^{k+p}.
\end{equation*}

The class of symbols we will be dealing with are those of the form $e^{ip\theta}\phi$ where $\phi(r)=r^{(2M+1)p}$, with $M\geq 1$ being  integer. It has been shown in \cite[Remark~15, ii)]{l} that in this case the root $T_{e^{i\theta}\widetilde{\phi}}$ exists and $\widetilde{\phi}$  is a polynomial in $r$ whose Mellin transform satisfies 
\begin{equation}\label{mellinroot}
\displaystyle{ \mathcal{M}\left(r\widetilde{\phi}\right)(z)=\frac{\prod\limits_{j=0}^{M-1}(z+2jp+2p)}{\prod\limits_{j=0}^{M}(z+2jp+2)}}, \textrm{ for }\Re z>2. 
\end{equation}
We are now ready to state and prove our main result.

\begin{thm}
Let $\phi(r)=r^{(2M+1)p}$ and $\psi(r)=r^{(2N+1)s}$ with $p<s$, $M$, and $N$ being all integers greater or equal to $1$. If there exist $m, l\in\mathbb{N}$ and nontrivial radial functions $f, g$ such that  $(H1)$ and $(H2)$ are satisfied, then $m=p,\ l=s$ and
$$T_{e^{im\theta}f}+T_{e^{il\theta}g}=c\left(T_{e^{ip\theta}\phi}+T_{e^{is\theta}\psi}\right)$$ for some constant $c$.
\end{thm}

\begin{proof} We shall keep the same notation introduced earlier that is $T_{e^{i\theta}\widetilde{\phi}}$ (resp. $T_{e^{i\theta}\widetilde{\psi}}$) is the root of $T_{e^{ip\theta}\phi}$ (resp.  $T_{e^{ip\theta}\psi}$). Since $(H1)$ and $(H2)$ are satisfied, equations (\ref{root_f}), (\ref{root_g}) and (\ref{com}) hold. Moreover, using (\ref{mellinroot}), Equation (6) becomes
\begin{equation}\label{Prod}
c_1\left(R_1P_1-R_2P_2\right)\\
=c_2\left(R_3P_3-R_4P_4\right), \textrm{ for all } k\geq 0
\end{equation}
where 
$$ R_1=\frac{k+s+1}{k+(N+1)s+1},\  \displaystyle{P_1=\prod_{j=1}^{M}\left[\frac{k+s+jp+1}{k+s+m+jp+1}\right]},$$
$$R_2= \frac{k+m+s+1}{k+m+(N+1)s+1},\ P_2=\prod_{j=1}^{M}\left[\frac{k+jp+1}{k+m+jp+1}\right],$$
$$R_3=\frac{k+l+p+1}{k+l+(M+1)p+1},\ \displaystyle{\  P_3=\prod_{j=1}^{N}\left[\frac{k+js+1}{k+l+js+1}\right]},$$ 
$$R_4=\frac{k+p+1}{k+(M+1)p+1},\  P_4=\prod_{j=1}^{N}\left[\frac{k+p+js+1}{k+l+p+js+1}\right].$$
The proof is mainly to show that Equation (\ref{Prod}) is true if and only if $m=p$ and so $l=s$ by $(H2)$. For the sufficiency, 
it is clear that if $m=p$ and so $l=s$, Equation (\ref{Prod}) is reduced to
\begin{eqnarray*}c_1\Big[\frac{k+s+1}{k+(N+1)s+1}\frac{k+s+p+1}{k+s+(M+1)p+1}-&\\\frac{k+p+s+1}{k+p+(N+1)s+1}\frac{k+p+1}{k+(M+1)p+1}\Big]=&\\c_2\Big[\frac{k+s+p+1}{k+s+(M+1)p+1}\frac{k+s+1}{k+(N+1)s+1}-&\\\frac{k+p+1}{k+(M+1)p+1}\frac{k+p+s+1}{k+p+(N+1)s+1}\Big] \textrm{ for all }k\geq 0.
\end{eqnarray*}
Obviously this is possible if and only if $c_1=c_2$, and therefore the desired result is obtained. To prove the necessity, we shall proceed by contradiction. We shall assume $m\neq p$ i.e., $l\neq s$ and we shall show that the set of poles of the left-hand side (LHS) of Equation (\ref{Prod}) is not equal to the set of poles of its right-hand side (RHS). To do so, either we exhibit one pole in the LHS (resp. RHS) and show that it does not appear in the RHS (resp. LHS) or we count the poles of each sides and show that the totals are distinct. For easiness, we will call pole any term in the denominator of $R_i$ or $P_i$. For example, $k+s+m+Mp+1$ is a pole obtained by taking $j=M$ in the denominator of $P_1$. 

The key of this proof by contradiction is the following: first we shall assume $m\neq tp$ for $t\geq 1$, and second $m=tp$ but for $t\geq 2$. Thus, in the end and after reaching the contradiction, we shall be left with the only possibility $m=p$.  
\begin{equation*}
\framebox[1\width]{\textrm{ {\bf{I.}} Assume } $m\neq tp$ \textrm{ for } $t\geq 1$ }
\end{equation*}
In this case none of the terms in the denominator of $P_2$ can be canceled by its numerator. Let us consider the pole $k+m+p+1$ obtained by taking $j=1$ in the denominator $P_2$. Obviously, this pole is not eliminated by the numerator of $R_2$ because $s>p$. We shall show that this pole does not appear in RHS:
\begin{itemize} 
\item[$\ast$] Can $k+m+p+1$ be equal to the pole $k+l+(M+1)p+1$  of $R_3$? In this case we have $m=l+Mp$ which is impossible because $m<l$.
\item[$\ast$] Can $k+m+p+1$ be equal to a pole $k+l+js+1$  of $P_3$ for some $1\leq j\leq N$? In this case $m+p=l+js$ which is not possible because $m<l$ and $p<s$.
\item[$\ast$] Can $k+m+p+1$ be equal to the pole $k+(M+1)p+1$  of $R_4$? In this case $m=Mp$ which contradicts our assumption that $m$ is not a multiple of $p$.
\item[$\ast$] Can $k+m+p+1$ be equal to a pole  $k+l+p+js+1$ of $P_4$ for some $1\leq j\leq N$? In this case $m=l+js$ which is not possible since $m<l$.
\end{itemize}
We conclude by saying that, under the assumption $m\neq tp$ for $t\geq 1$, Equation (\ref{Prod}) cannot be satisfied since we were able the find a pole from LHS that does not appear in RHS.
\begin{equation*}
\framebox[1\width]{\textrm{ {\bf{II.}} Assume } $m=tp$ \textrm{ for } $t\geq 2$ }
\end{equation*}
We shall discuss two situations: "$M\leq t$" and "$M>t$". In fact when $M>t$, terms of the numerators of $P_1$ (resp. $P_2$)  would cancel some poles of the denominator. This is not the case when $M\leq t$.
\begin{itemize}
\item{\underline{$1^{st}$ Situation: $M\leq t$.}} Consider the pole $k+m+p+1=k+(t+1)p+1$ of $P_2$ obtained by taking $j=1$. Since $M\leq t$, this pole is not canceled by the numerator of $P_2$. Also, since $p<s$ the same pole is not canceled by the numerator of $R_2$. Moreover, it is easy to see that $k+m+p+1$ is equal to none of the poles of $R_3$, $P_3$ and $P_4$ because $m<l$ and $p<s$. We still have to check if $k+m+p+1$ can be equal to the pole  $k+(M+1)p+1$ of $R_4$. In this case $m=Mp$ i.e., $t=M$. Since $t\geq 2$, the denominator of $P_2$ will have at least two poles. Consider the pole $k+m+2p+1$ of $P_2$ obtained by taking $j=2$:
\begin{itemize}
	\item[$\ast$] When this pole is not canceled by the numerator of $R_2$ i.e., $s\neq 2p$, it is easy to see that $k+m+2p+1$ is not equal to any of the poles of RHS because $m<l$ and $p<s$.
	\item[$\ast$] Assume $k+m+2p+1$ is canceled by the numerator of $R_2$ i.e., $s=2p$. In this case $l=m+s-p=m+p=(M+1)p$. Then:
	\begin{itemize}
		\item[$\star$] If $M=2$: the pole $k+m+p+1=k+3p+1$ of $P_2$ obtained by taking $j=1$ does not appear in RHS because the pole of $R_4$ is  eliminated by the first term of the numerator of $P_4$ (since $s=2p$) and this for any $N\geq 1$.
		\item[$\star$] If $M\geq 3$: the pole $k+Mp+3p+1$ appears at least twice in LHS when taking $j=1$ in the denominator of $P_1$ and $j=3$ in the denominator of $P_2$. However this same pole appears at most once in RHS when taking $j=1$ in the denominator of $P_3$. 
	\end{itemize}
\end{itemize} 
Therefore, $k+m+p+1$ cannot be equal to $k+(M+1)p+1$, and hence we conclude by saying that under the assumption $m=tp$ for $t\geq 2$ and when $M\leq t$, Equation (\ref{RP}) is not satisfied since there exists a pole of LHS that is not a pole of RHS.
\item{\underline{$2^{nd}$ Situation: $M>t\geq 2$.}} In this case poles of $P_1$ (resp. $P_2$) are canceled by terms of its numerator, and Equation (\ref{Prod}) becomes
\begin{multline}\label{RP}
c_1\Big[\frac{k+s+1}{k+(N+1)s+1}\frac{\prod\limits_{j=1}^{t}k+s+jp+1}{\prod\limits_{j=M-t+1}^{M}k+s+tp+jp+1}-\\ \frac{k+tp+s+1}{k+tp+(N+1)s+1}\frac{\prod\limits_{j=1}^{t}k+jp+1}{\prod\limits_{j=M-t+1}^{M}k+tp+jp+1}\Big]=\\	
c_2\Big[\frac{k+l+p+1}{k+l+(M+1)p+1}\prod_{j=1}^{N}\frac{k+js+1}{k+l+js+1}-\\\frac{k+p+1}{k+(M+1)p+1}\prod_{j=1}^{N}\frac{k+p+js+1}{k+l+p+js+1}\Big].	
\end{multline}
It is very important to notice here that LHS has at most $2t+2$ poles and at least $2t$ poles. In fact, the pole of $R_1$ could be eliminated by the numerator of $P_1$ and the numerator of $R_2$ could cancel a pole of $P_2$. However RHS has at most $2N+2$ poles. Thus, if $N<t-1$ Equation (\ref{RP}) cannot be satisfied. Therefore we must have $N\geq t-1$.\\
Now, consider the pole $k+s+tp+(M-t+1)p+1=k+s+(M+1)p+1$ of $P_1$ obtained by taking $j=M-t+1$. Clearly it is not eliminated by the numerator of $R_1$. We shall show that this pole does not appear in RHS:
\begin{itemize}
	 
\item[$\ast$] Can $k+s+(M+1)p+1$ be equal to the pole $k+l+(M+1)p+1$ of $R_3$? In this case $s=l$ and so $m=p$, which contradicts our assumption that $m=tp$ for $t\geq 2$.

\item[$\ast$] Can $k+s+(M+1)p+1$ be equal to the pole $k+(M+1)p+1$ of $R_4$? Clearly no.

\item[$\ast$] Can $k+s+(M+1)p+1$ be equal to a pole $k+l+js+1$ of $P_3$ for some $1\leq j\leq N$? In this case $js=(M-t+2)p$. Since $s>p$, we must have $j<M-t+2$. Let us denote such $j$ by $j^*$ i.e, $1\leq j^*\leq N$, $j^*<M-t+2$ and $j^*s=(M-t+2)p$. We shall discuss the following cases:
\begin{itemize}
	
	\item[$\star$] ${\underline{M=3}}$. Then $t=2$, $M-t+2=3$ and so $j^*=1 \textrm{ or }2$.
	
		\begin{itemize}
				
		 \item[(i)] If $j^*=1$, then $s=3p$ and $l=m+s-p=4p$. Thus Equation (\ref{RP}) becomes
		  \begin{multline*}
              c_1\Big[\frac{k+3p+1}{k+3(N+1)p+1}\frac{(k+4p+1)(k+5p+1)}{(k+7p+1)(k+8p+1)}-\\
              \frac{1}{k+(3N+5)p+1}\frac{(k+p+1)(k+2p+1)}{k+4p+1}\Big]=\\
              c_2\Big[\frac{k+5p+1}{k+8p+1}\prod_{j=1}^{N}\frac{k+3jp+1}{k+(4+3j)p+1}-\frac{k+p+1}{k+4p+1}\prod_{j=1}^{N}\frac{k+(1+3j)p+1}{k+(5+3j)p+1}\Big].	
		\end{multline*}
		It is easy to see that the pole $k+4p+1$ of $R_4$ is canceled by the first term of the numerator of $P_4$. Hence this pole appears in LHS, as a pole of $P_2$, but does not appear in RHS. Therefore $j^*$ cannot be $1$.
		\item[(ii)] $j^*=2$, then $2s=3p$ and $l=p+s$. Equation (\ref{RP}) becomes
				\begin{multline*}
					c_1\Big[\frac{k+s+1}{k+(N+1)s+1}\frac{(k+s+p+1)(k+s+2p+1)}{(k+s+4p+1)(k+s+5p+1)}-\\
					\frac{k+2p+s+1}{k+2p+(N+1)s+1}\frac{(k+p+1)(k+2p+1)}{(k+4p+1)(k+5p+1)}\Big]=\\
					c_2\Big[\frac{k+2p+s+1}{k+s+5p+1}\prod_{j=1}^{N}\frac{k+js+1}{k+p+(j+1)s+1}-\\\frac{k+p+1}{k+4p+1}\prod_{j=1}^{N}\frac{k+p+js+1}{k+2p+(j+1)s+1}\Big].	
				\end{multline*}
				We observe that when $N\geq 3$ the pole $k+s+5p+1$ appears only once in LHS from the denominator of $P_1$, however the same pole appears at least twice in RHS from the denominator of $R_3$ and from taking $j=2$ in the denominator of $P_4$ ($3s=s+2s=s+3p$). Note that this pole is not eliminated by either the numerator of $P_4$ or $R_4$. Now, when $N=2$ ($N$ cannot be $1$ because $N\geq j^*$ and $j^*=2$), Equation (\ref{RP}) becomes 
				\begin{multline*}
					c_1\Big[\frac{k+s+1}{k+3s+1}\frac{(k+s+p+1)(k+s+2p+1)}{(k+s+4p+1)(k+s+5p+1)}-\\
					\frac{k+s+2p+1}{k+2p+3s+1}\frac{(k+p+1)(k+2p+1)}{(k+4p+1)(k+5p+1)}\Big]=\\
					c_2\Big[\frac{k+2p+s+1}{k+s+5p+1}\frac{(k+s+1)(k+2s+1)}{(k+p+2s+1)(k+p+3s+1)}-\\\frac{k+p+1}{k+4p+1}\frac{(k+p+s+1)(k+p+2s+1)}{(k+2p+2s+1)(k+2p+3s+1)}\Big].
				\end{multline*}
				By equating the poles in both sides, we must have $k+3s+1=k+p+2s+1$ i.e., $s=p$ which is impossible because $s>p$.
	\end{itemize}
    We conclude by saying that, when $M=3$ and $j^*s=(M-t+2)p$, the pole $k+s+(M+1)p+1$ cannot be equal to the pole $k+l+js+1$ of $P_3$ for some $1\leq j\leq N$. 
 
 \item[$\star$] $\underline{M\geq 4}$. Here we shall discuss two cases "$t=2$" and $"t\geq 3"$. In fact when $t\geq 3$ the denominator of $P_2$ contains at least $3$ poles one of them is $k+m+(M-t+3)p+1=k+m+j^*s+p+1$ obtained by taking $j=M-t+3$.
  
   \begin{itemize}
	\item[(i)]{$t\geq 3$}.  We shall show that the pole $k+m+j^*s+p+1$ of $P_2$ does not appear in RHS:\\
	$\circ$ Can $k+m+j^*s+p+1$ be equal to the pole $k+l+(M+1)p+1$ of $R_3$? In this case $(j^*-1)s= (M-1)p$. But since $j^*s=(M-t+2)p$, it implies that $s=(3-t)p$ which is impossible because $t\geq 3$.\\
	$\circ$ Can $k+m+j^*s+p+1$ be equal to a pole $k+l+js+1$ of $P_3$ for some $1\leq j\leq N$? In this case $2p=(1+j-j^*)s$, which is possible only if $j=j^*$ i.e., $s=2p$. Thus Equation (\ref{RP}) becomes
	\begin{multline*}
		c_1\Big[\frac{k+2p+1}{k+2(N+1)p+1}\frac{\prod\limits_{j=1}^{t}k+(2+j)p+1}{\prod\limits_{j=M-t+1}^{M}k+(2+t+j)p+1}-\\\frac{k+(t+2)p+1}{k+tp+2(N+1)p+1}\frac{\prod\limits_{j=1}^{t}k+jp+1}{\prod\limits_{j=M-t+1}^{M}k+(t+j)p+1}\Big]=\\	
		c_2\Big[\frac{k+(t+2)p+1}{k+(t+M+2)p+1}\prod_{j=1}^{N}\frac{k+2jp+1}{k+tp+p+js+1}-\\\frac{k+p+1}{k+(M+1)p+1}\prod_{j=1}^{N}\frac{k+(1+2j)p+1}{k+(t+2)p+js+1}\Big].	
	\end{multline*}
	Now we observe that the pole $k+(M+3)p+1$ appears twice in LHS from $j=M-t+1$ in the denominator of $P_1$ and also from $j=M-t+3$ in the denominator of $P_2$. However this same pole appears only once in RHS when $j=j^*$ in the denominator of $P_3$. In fact  $k+tp+p+j^*s+1=k+tp+p+(M-t+2)p+1=k+(M+3)p+1$. Moreover $k+(M+3)p+1$ cannot be equal to $k+(t+2)p+js+1$ for some $1\leq j\leq N$ because, knowing that $j^*s=(M-t+2)p$, we would have $(j^*-j)s=p$ which is impossible since $s>p$. Therefore $k+m+j^*s+p+1$ cannot be equal to $k+l+js+1$ for some $1\leq j\leq N$.\\
	$\circ$ Can $k+m+j^*s+p+1$ be equal to $k+(M+1)p+1$? Obviously not because $k+m+j^*s+p+1=k+(M+3)p+1$.\\
	$\circ$ Can $k+m+j^*s+p+1$ be equal to $k+l+p+js+1$ for some $1\leq j\leq N$? In this case $p=(1+j-j^*)s$ which is impossible because $s>p$.\\
	We conclude by saying that, when $M\geq 4$, $t\geq 3$, and $j^*s=(M-t+2)p$, the pole $k+m+j^*s+p+1$ does not appear in RHS.
	
	\item[(ii)] $t=2$. In this case LHS has at most $6$ poles and at least $5$ poles. In fact the pole of $R_1$ is canceled by the numerator of $P_1$ when $N=1$ and $s=2p$. On the other hand, RHS has at most $2N+2$ poles, and so Equation (\ref{RP}) can't be satisfied if $N=1$. Moreover, since $m=tp=2p$, we have that $l=p+s$ and thus none of the terms of  the numerator of $P_3$ (resp. $P_4$) can cancel the poles of $P_3$ (resp. $P_4$). Therefore if $N\geq 4$, then Equation (\ref{RP}) is not satisfies because RHS has at least $2N\geq 8$ poles. Finally, we have to check the cases $"N=2"$ and $"N=3"$:\\
	$\circ$ $N=2$. Equation (\ref{RP}) becomes
		\begin{multline*}
			c_1\Big[\frac{k+s+1}{k+3s+1}\frac{(k+s+p+1)(k+s+2p+1)}{\left(k+s+(M+1)p+1\right)\left(k+s+(M+2)p+1\right)}-\\
			\frac{k+2p+s+1}{k+2p+3s+1}\frac{(k+p+1)(k+2p+1)}{\left(k+(M+1)p+1\right)\left(k+(M+2)p+1\right)}\Big]=\\
			c_2\Big[\frac{k+2p+s+1}{k+s+(M+2)p+1}\frac{(k+s+1)(k+2s+1)}{(k+p+2s+1)(k+p+3s+1)}-\\\frac{k+p+1}{k+(M+1)p+1}\frac{(k+p+s+1)(k+p+2s+1)}{(k+2p+2s+1)(k+2p+3s+1)}\Big].	
		\end{multline*}
		Now, it is not hard to see that when trying to equate the poles from both sides we reach one of the following contradiction namely either $"2s=Mp\textrm{ and }2s=(M+1)p"$ or $"s=p"$.\\
	$\circ$ $N=3$. Equation (\ref{RP}) becomes
			\begin{multline*}
				c_1\Big[\frac{k+s+1}{k+4s+1}\frac{(k+s+p+1)(k+s+2p+1)}{\left(k+s+(M+1)p+1\right)\left(k+s+(M+2)p+1\right)}-\\
				\frac{k+2p+s+1}{k+2p+4s+1}\frac{(k+p+1)(k+2p+1)}{\left(k+(M+1)p+1\right)\left(k+(M+2)p+1\right)}\Big]=\\
				c_2\Big[\frac{k+2p+s+1}{k+s+(M+2)p+1}\frac{(k+s+1)(k+2s+1)(k+3s+1)}{(k+p+2s+1)(k+p+3s+1)(k+p+4s+1)}-\\\frac{k+p+1}{k+(M+1)p+1}\frac{(k+p+s+1)(k+p+2s+1)(k+p+3s+1)}{(k+2p+2s+1)(k+2p+3s+1)(k+2p+4s+1)}\Big].	
			\end{multline*}
			This equality is possible only if the poles $k+s+(M+2)p+1$ and $k+(M+1)p+1$ in RHS are both canceled so that the numbers of poles on both sides are equal. Notice that these are the only two poles that might be canceled in RHS. Thus we must have $"s=(M+2)p \textrm { or } 2s=(M+2)p"$ and $"s=Mp\textrm{ or } 2s=Mp\textrm { or }3s=Mp"$. The only possible combination is to have $"s=2p\textrm{ and }M=2"$. But this can't be true because $M\geq 4$. 
  \end{itemize}
  
 Hence when $M\geq 4$, $t=2$ and $j^*s=(M-t+2)p$, Equation (\ref{RP}) cannot be satisfied.
\end{itemize} 
	
Therefore, we conclude by saying that the pole $k+s+(M+1)p+1$ cannot be equal to $k+l+js+1$ for some $1\leq j\leq N$ in all possible situations.	

\item[$\ast$] Can $k+s+(M+1)p+1$ be equal to a pole $k+l+p+js+1$ of $P_4$ for some $1\leq j\leq N$? In this case $js=(M-t+1)p$ for some $1\leq j\leq N$. Let us denote such $j$ by $j^*$. Observe that $j^*<(M-t+1)$ because $s>p$. We shall discuss the following cases:

\begin{itemize}
	\item[$\star$] ${\underline{M=3}}$. Then $t=2$ and so $j^*=1$. Moreover we have $m=2p$, $s=2p$ and $l=3p$. Thus Equation(\ref{RP}) becomes
	\begin{multline*}
		c_1\Big[\frac{k+2p+1}{k+2(N+1)p+1}\frac{(k+3p+1)(k+4p+1)}{(k+6p+1)(k+7p+1)}-\\
		\frac{1}{k+2(N+2)p+1}\frac{(k+p+1)(k+2p+1)}{k+5p+1}\Big]=\\
		c_2\Big[\frac{k+4p+1}{k+7p+1}\prod_{j=1}^{N}\frac{k+2jp+1}{k+(3+2j)p+1}-\frac{k+p+1}{k+4p+1}\prod_{j=1}^{N}\frac{k+(1+2j)p+1}{k+(4+2j)p+1}\Big].	
	\end{multline*} 
	It is easy to see that when $N=1$ the pole $k+4p+1$ of $R_1$ is canceled by the numerator of $P_1$ and so it does not appear in LHS. However it appears in RHS as a pole of $R_4$. Now, when $N\geq 2$, LHS has 5 poles but RHS has $2N +2$ poles. Therefore, we conclude that when $M=3$ we cannot have $k+s+(M+1)p+1$ equal to $k+l+p+js+1$ for some $1\leq j\leq N$.
	\item[$\star$] ${\underline{M\geq4}}$. Consider the pole $k+tp+(M-t+2)p+1=k+m+j^*s+p+1$ of $P_2$ obtained by taking $j=M-t+2$. We shall show that this pole does not appear in RHS. 
	\begin{itemize}
		\item[(i)] Can $k+m+j^*s+p+1$ be equal to the pole $k+l+(M+1)p+1$ of $R_3$? In this case $(j^*-1)s=(M-1)p$. But since $j^*s=(M-t+1)p$, we must have $s=(2-t)p$ which is impossible because $t\geq 2$.
		\item[(ii)] Can $k+m+j^*s+p+1$ be equal to the pole $k+(M+1)p+1$ of $R_4$? Obviously not because we would have $2p=p$!
		\item[(iii)] Can $k+m+j^*s+p+1$ be equal to a pole $k+l+p+js+1$ of $P_4$ for some $1\leq j\leq N$? In this case $(j+1-j^*)s=p$ which is not possible because $s>p$.
		\item[(iv)] Can $k+m+j^*s+p+1$ be equal to a pole $k+l+js+1$ of $P_3$ for some $1\leq j\leq N$? In this case $(j+1-j^*)s=2p$. This is possible only when $j=j^*$ i.e., $s=2p$. Here, we shall make the distinction between two cases $"t \textrm{ even}"$ and $"t \textrm{ odd}"$:\\
		$\circ$ If $t=2$ and $N=1$, then LHS has five poles and RHS has four poles. Hence, Equation (\ref{RP}) cannot be satisfied.\\
		$\circ$ If $t$ is even and $N\geq 2$, then we have the following cases: First, when $"M=4,\  t=2"$ (because $t\leq M-1$ and $t$ is even) LHS has exactly six poles, however RHS has either five (when $N=2$) or seven (when $N=3$) or $2N$ (when $N\geq 4$) poles. Therefore Equation (\ref{RP}) is not satisfied. Second, when $"M\geq 5,\  M-2\leq t \leq M-1"$ the pole $k+(t+5)p+1$ appears twice in LHS by taking $j=3$ in the denominator of $P_1$ and $ j=5$ in the denominator of $P_2$, however the same pole appears only once in RHS by taking $j=2$ in the denominator of $P_3$. Hence Equation (\ref{RP}) cannot be satisfied. Finally, when  $"M\geq 5,\ t \leq M-3"$ the pole $k+l+s+1=k+(t+3)p+1$ of $P_3$ (notice that this pole cannot be eliminated by either the numerator of $R_3$ or the numerator of $P_3$) obtained by taking $j=1$ does not appear in LHS since $M-t+1\geq 4$. Thus Equation (\ref{RP}) is again not satisfied.\\
		We conclude by saying that when $t$ is even $k+m+j^*s+p+1$ cannot be equal to $k+l+js+1$ for some $1\leq j\leq N$.\\
		$\circ$ If $t$ is odd (here $t$ must be greater or equal to $3$ because we are assuming $m=tp$ for $t\geq 2$), then Equation (\ref{RP}) becomes
		\begin{multline*}
			c_1\Big[\frac{k+2p+1}{k+2(N+1)p+1}\frac{\prod\limits_{j=1}^{t}k+(2+j)p+1}{\prod\limits_{j=M-t+1}^{M}k+(2+t+j)p+1}-\\\frac{k+(t+2)p+1}{k+tp+2(N+1)p+1}\frac{\prod\limits_{j=1}^{t}k+jp+1}{\prod\limits_{j=M-t+1}^{M}k+(t+j)p+1}\Big]=\\	
			c_2\Big[\frac{k+(t+2)p+1}{k+(t+M+2)p+1}\frac{\prod\limits_{j=1}^{\frac{t+1}{2}}k+2jp+1}{\prod\limits_{j=N-\frac{t+1}{2}+1}^{N}k+(t+1+2j)p+1}-\\\frac{k+p+1}{k+(M+1)p+1}\frac{\prod\limits_{j=1}^{\frac{t+1}{2}}k+(1+2j)p+1}{\prod\limits_{j=N-\frac{t+1}{2}+1}^{N}   k+(t+2+2j)p+1}\Big].
		\end{multline*}
		It is easy to see that LHS has at least $2t+1$ poles (in fact if $t=M-1$, then the first pole of $P_2$ obtained by taking $j=M-t+1=2$ is canceled by the numerator of $R_2$), however RHS has at most $t+3$ poles. Now, $2t+1=t+3$ if and only if $t=2$ which is impossible because we are assuming $t$ is odd. Hence, when $t$ is odd $k+m+j^*s+p+1$ cannot be equal to $k+l+js+1$ for some $1\leq j\leq N$. \\
		We conclude by saying that the pole $k+m+j^*s+p+1$ does not appear in RHS.
		\end{itemize}
	\end{itemize}
This finishes proving that the pole $k+s+(M+1)p+1$ cannot be equal to $l+p+js+1$ for some $1\leq j\leq N$.
\end{itemize} 
Therefore the pole $k+s+(M+1)p+1$ from LHS does not appear in RHS.
\end{itemize}
We conclude by saying that, under the assumption $m=tp$ for $t\geq 2$, Equation (\ref{RP}) cannot be satisfied since we always were able the find a pole from LHS that does not appear in RHS.\\

Therefore Equation (\ref{Prod}) is true if and only if $m=p$, which implies $l=s$ by $(H2)$, and so $c_1=c_2$ by Remark 1, iii). Finally, we obtain the desired result namely
\begin{equation*}
T_{e^{im\theta}f}+T_{e^{il\theta}g}=c\left(T_{e^{ip\theta}\phi}+T_{e^{is\theta}\psi}\right)
\end{equation*}
\end{proof}

\begin{remark}
\item[i)] Theorem 1 can be generalized to sums of more than two Toeplitz operators not without extra efforts in the proof.
\item[ii)] If $m_1, l_1, m_2, l_2, f_1, g_1, f_2, g_2$ are as in Theorem 1, then we will have
\begin{align*}
T_{e^{im_1\theta}f_1}+T_{e^{il_1\theta}g_1}=c\left(T_{e^{ip\theta}\phi}+T_{e^{is\theta}\psi}\right),\end{align*} and 
\begin{align*}
T_{e^{im_2\theta}f_2}+T_{e^{il_2\theta}g_2}=c'\left(T_{e^{ip\theta}\phi}+T_{e^{is\theta}\psi}\right).
\end{align*}
for some constants $c_1, c_2$, and therefore $T_{e^{im_1\theta}f_1}+T_{e^{il_1\theta}g_1}$ and $T_{e^{im_2\theta}f_2}+T_{e^{il_2\theta}g_2}$ commute with each other. Thus, Theorem 1 is a partial confirmation of the conjecture in \cite[p.~263]{lr} which states that if two Toeplitz operators, defined on the Bergman space of the unit disc $\mathbb{D}$, commute with a third one, none of them being the identity, then they commute with each other.
	
\end{remark}


{\small
\begin{thebibliography}{9}
\bibitem{ac} S. Axler and \u{Z}. \u{C}u\u{c}kovi\'{c}, Commuting Toeplitz operators with harmonic symbols, \emph{Integral Equations and Operator Theory} 14, 1–12 (1991).
\bibitem{acr} S. Axler, \u{Z}. \u{C}u\u{c}kovi\'{c}, and N. V. Rao, Commutants of analytic Toeplitz operators on the Bergman space, \emph{Proc. Amer. Math. Soc.} 128, 1951–1953 (2000).	
\bibitem{c} \u{Z}. \u{C}u\u{c}kovi\'{c}, Commutants of Toeplitz operators on the Bergman space, \emph{Pacific J. Math.} 162, 277–285 (1994).
\bibitem{cr}  \u{Z}. \u{C}u\u{c}kovi\'{c}, N. V. Rao,  Mellin transform, monomial symbols, and commuting Toeplitz operators. \emph{J. Funct. Anal.} 154 (1998), no. 1, 195-214.
\bibitem{l} I. Louhichi, Powers and roots of Toeplitz operators, \emph{Proc. Amer. Math. Soc.} {\bf{135}}:5 (2007), 1465-1475.
\bibitem{lr} I. Louhichi and N. V. Rao, Bicommutants of Toeplitz operators, \emph{Arch. Math.} {\bf{91}}, (2008), 256-264.
\bibitem{lr2}  I. Louhichi and N. V. Rao, Roots of Toeplitz operators on the Bergman space, \emph{ Pacific J. Math.} 252 (2011), no. 1, 127–144.
\bibitem{lra} I. Louhichi, N. V. Rao, A. Yousef, Two questions on the theory of Toeplitz operators
on the Bergman space, \emph{Complex Anal. Oper. Theory} 3 (2009), no. 4, 881-889.
\bibitem{lsz} I. Louhichi, E. Strouse, L. Zakariasy, Products of Toeplitz operators on the Bergman space, \emph{Integral Equations Operator Theory} 54 (2006), 525-539.
\bibitem {lz} I. Louhichi, L. Zakariasy, On Toeplitz operators with quasihomogeneous symbols, \emph{Arch. Math.} {\bf {85}} (2005), 248-257.
\bibitem{v} N. Vasilevski, Bergman space structure, commutative algebras of Toeplitz operators and hyperbolic geometry, \emph{Integral Equations and Operator Theory} 46, 235–251 (2003).
\end{thebibliography}}
\end{document}